\documentclass[a4paper,11pt]{article}

\usepackage[multidot]{grffile}
\usepackage{latexsym,amssymb,enumerate,amsmath,epsfig,amsthm}
\usepackage[margin=1in]{geometry}
\usepackage{setspace,color}
\usepackage{graphicx}
\usepackage[ruled]{algorithm2e}
\usepackage{empheq}
\usepackage{bm}
\usepackage{subcaption}
\usepackage{booktabs,array,multirow}
\usepackage{hyperref}
\hypersetup{
    colorlinks=true,
    linkcolor=blue,
    filecolor=magenta,
    urlcolor=cyan,
    citecolor=red,
    pdfpagemode=FullScreen,
}

\newcolumntype{C}{>{$\displaystyle}c<{$}}

\DeclareMathOperator*{\argmin}{argmin}

\newcommand{\bR}{\mathbb{R}}

\newcommand{\Log}{\operatorname{Log}}
\newcommand{\Exp}{\operatorname{Exp}}
\newcommand{\SLERP}{\operatorname{SLERP}}
\newcommand{\SIDER}{\operatorname{SIDER}}

\DeclareMathOperator{\Gr}{Gr}
\DeclareMathOperator{\St}{St}
\DeclareMathOperator{\spanop}{span}
\DeclareMathOperator{\diag}{diag}
\DeclareMathOperator{\tr}{tr}

\newcommand{\cY}{\mathcal{Y}}

\newcommand{\norm}[1]{\left\lVert #1 \right\rVert}
\newcommand{\inner}[2]{\left\langle #1,#2 \right\rangle}

\newcommand{\GLERP}{\operatorname{GLERP}}
\newcommand{\GIDER}{\operatorname{GIDER}}

\numberwithin{equation}{section}

\theoremstyle{plain}
\newtheorem{thm}{Theorem}[section]
\newtheorem{prop}[thm]{Proposition}
\newtheorem{lemma}[thm]{Lemma}

\theoremstyle{definition}
\newtheorem{definition}[thm]{Definition}
\newtheorem{assump}[thm]{Assumption}

\theoremstyle{remark}
\newtheorem{remark}[thm]{Remark}

\title{Geodesic Interpolation on the Grassmann Manifold:\\ GLERP and Recursive GIDER Interpolants}
\author{
Shingyu Leung\thanks{Department of Mathematics, the Hong Kong University of Science and Technology, Clear Water Bay, Hong Kong. Email: {\bf masyleung@ust.hk}}
}
\date{}

\begin{document}

\maketitle

\begin{abstract}
This article develops a geodesic interpolation framework for data on the Grassmann manifold.  The motivation is that many matrix-valued data sets represent subspaces rather than fixed bases: if the columns of two matrices differ only by a right orthogonal transformation, then they describe the same point on \(\Gr(r,m)\).  Interpolation should therefore be invariant under this basis ambiguity.  We first introduce \(\GLERP\), a Grassmann analogue of spherical linear interpolation defined by the Grassmann exponential and logarithm maps.  The method follows the constant-speed geodesic joining two nearby subspaces and is second-order accurate for smooth Grassmann-valued curves under the usual normal-neighborhood condition.  We then define \(\GIDER_n\), a recursive higher-order construction obtained by replacing each affine interpolation step in Neville's algorithm by \(\GLERP\).  The resulting interpolant matches \(n+1\) subspace data exactly, is basis-invariant, and is locally of order \(n+1\) for sufficiently smooth curves.  We compare the construction with tangent-space interpolation and projection-matrix interpolation, discuss intrinsic and extrinsic error measures, and present numerical tests.  The results confirm the expected convergence orders and show that \(\GIDER_n\) and tangent-space interpolation are nearly indistinguishable in a smooth local regime, while projection-matrix interpolation provides a useful extrinsic baseline.  We also outline how the recursive construction can be used as the basis of a Grassmann ENO procedure.
\end{abstract}

\section{Introduction}

Interpolation of constrained or geometry-valued data arises in many areas of scientific computing, numerical analysis, computer vision, data analysis, and reduced-order modeling.  When the data lie on a nonlinear manifold, direct interpolation in an ambient Euclidean space may destroy the relevant geometric structure.  A familiar example is interpolation on the unit sphere, where spherical linear interpolation, or \(\SLERP\), provides the constant-speed shortest geodesic between two points.  In the spherical interpolation framework of Fong and Leung \cite{FongLeung2023}, \(\SLERP\) is used as a geometric replacement for Euclidean linear interpolation and is incorporated into higher-order and essentially non-oscillatory constructions.  The guiding principle is that a Euclidean interpolation procedure can often be transferred to a nonlinear manifold by replacing affine segments with geodesic segments.

The present work develops this principle for the Grassmann manifold \(\Gr(r,m)\), the space of all \(r\)-dimensional linear subspaces of \(\bR^m\).  This setting is natural when the data are matrices whose column spaces, rather than their particular entries, carry the relevant information.  If \(Y\in\bR^{m\times r}\) has orthonormal columns, then \(Y\) and \(YR\), for any \(R\in O(r)\), represent the same subspace.  The object to be interpolated is therefore not the matrix representative itself, but the Grassmann point \(\spanop(Y)\).  Any meaningful interpolation formula and any meaningful error measure should be invariant under the right orthogonal action \(Y\mapsto YR\).  This rules out naive entrywise interpolation of matrix representatives as a generally valid Grassmann interpolation procedure.

Several interpolation strategies for Grassmann-valued data are already available or arise naturally from the geometry.  A widely used approach is tangent-space interpolation, in which the data are mapped by a logarithm map to the tangent space at a selected reference subspace, interpolated there by ordinary polynomial interpolation, and mapped back by the exponential map.  This idea is common in parametric reduced-order modeling and related applications \cite{AmsallemFarhat2011}.  Another basis-invariant approach represents each subspace by its orthogonal projector \(P=YY^T\), interpolates the projectors in the ambient vector space of symmetric matrices, and then projects the result back to \(\Gr(r,m)\) by taking the leading \(r\) eigenvectors.  A third family of methods is based on recursive geodesic constructions, including manifold versions of Neville--Aitken interpolation \cite{NoakesHeinzingerPaden1989,MosqueraEtAl2019}.  These methods are closely related in a small smooth regime, but they differ in their use of tangent coordinates, their dependence on auxiliary choices, their numerical cost, and their behavior near the boundary of a normal neighborhood.  A detailed comparison of the three approaches considered in this work is given in Section~\ref{sec:comparison_three_methods}.

The first contribution of this article is to define a Grassmann analogue of \(\SLERP\), which we call \(\GLERP\).  For two nearby Grassmann points \(\mathcal{Y}_0,\mathcal{Y}_1\in\Gr(r,m)\), it is defined by
\[
    \GLERP(\mathcal{Y}_0,\mathcal{Y}_1;\tau)
    =
    \Exp_{\mathcal{Y}_0}
    \bigl(
        \tau\Log_{\mathcal{Y}_0}(\mathcal{Y}_1)
    \bigr),
    \qquad 0\le \tau\le 1.
\]
When the logarithm map is uniquely defined, this is the constant-speed geodesic joining the two subspaces.  In normal coordinates, \(\GLERP\) reduces to ordinary linear interpolation along a tangent vector.  We prove that \(\GLERP\) is second-order accurate for smooth Grassmann-valued curves under the standard normal-neighborhood assumption.

The second contribution is a recursive family of higher-order Grassmann interpolants, denoted by \(\GIDER_n\).  The construction follows Neville interpolation: starting from \(n+1\) data subspaces, each affine interpolation step is replaced by \(\GLERP\).  In particular, \(\GIDER_2\) is obtained by applying \(\GLERP\) to two lower-order \(\GLERP\) curves and may be viewed as a Grassmann analogue of quadratic Lagrange interpolation.  More generally, \(\GIDER_n\) interpolates \(n+1\) Grassmann data points exactly and reduces to the classical polynomial interpolant when the manifold is replaced by a Euclidean vector space.  This construction is distinct from the control-point form of \(\SIDER_2\) in \cite{FongLeung2023}; here the emphasis is on a recursive Lagrange--Neville type interpolation on \(\Gr(r,m)\).

The construction is local.  The Grassmann logarithm is single-valued only in an injectivity neighborhood, and the recursive formulas may involve geodesic extrapolation.  For the canonical Grassmann metric, this admissibility condition can be monitored by principal angles: the relevant subspaces, including intermediate subspaces generated by the recursion, should remain bounded away from configurations with a principal angle equal to \(\pi/2\).  We give below a quantitative sufficient condition showing that, for fixed interpolation order, the principal angles appearing in all recursive \(\GLERP\) steps are \(\mathcal{O}(h)\) when the sampled data are \(\mathcal{O}(h)\)-close to a reference subspace.  Thus sufficiently dense sampling keeps the construction uniformly away from the cut locus.

The main analytical result is a local accuracy theorem.  Under the normal-neighborhood assumption, \(\GIDER_n\) is locally \((n+1)\)-st order accurate for \(C^{n+1}\) Grassmann-valued curves.  The proof is a manifold version of the classical interpolation-error argument.  In a local coordinate chart, the recursive interpolant is smooth, agrees with the exact curve at \(n+1\) interpolation nodes, and the coordinate error is controlled by the product of the nodal factors.  The final estimate is stated intrinsically in terms of the Grassmann geodesic distance.  We also discuss projection-matrix errors and Procrustes-aligned errors, both of which are invariant under changes of orthonormal bases.

The rest of the paper is organized as follows.  Section~\ref{sec:background} recalls the necessary geometry of the Grassmann manifold, including principal angles, geodesic distance, and the exponential and logarithm maps.  Section~\ref{sec:glerp} introduces \(\GLERP\) and proves its second-order accuracy.  Section~\ref{sec:gider} defines the recursive \(\GIDER_n\) construction and establishes its interpolation property.  Section~\ref{sec:accuracy} proves the local \((n+1)\)-st order accuracy theorem.  Section~\ref{sec:error} discusses basis-invariant error measurements.  Section~\ref{sec:comparison_three_methods} compares \(\GIDER_n\) with tangent-space interpolation and projection-matrix interpolation.  Section~\ref{sec:numerical_example} presents numerical tests, Section~\ref{sec:geno} outlines a Grassmann ENO extension, and Section~\ref{sec:conclusion} concludes the paper.

\section{Background}\label{sec:background}

The purpose of this section is to recall the basic geometric ingredients needed for interpolation on the Grassmann manifold.  Unlike Euclidean interpolation, where data points can be added and scaled directly, interpolation on \(\Gr(r,m)\) must respect the fact that each point represents an \(r\)-dimensional subspace of \(\bR^m\).  Thus a matrix \(Y\in\bR^{m\times r}\) with orthonormal columns is only one representative of the subspace \(\spanop(Y)\), and the representatives \(Y\) and \(YR\), \(R\in O(r)\), are equivalent.  Any meaningful interpolation formula and error measure should therefore be invariant under this change of basis.  We summarize below the standard representation of tangent vectors, the Grassmann distance through principal angles, and the exponential and logarithm maps.  These tools provide the intrinsic replacement for vector addition and scalar multiplication and will be used to define the Grassmann linear interpolation \(\GLERP\) and its recursive higher-order extensions.

\subsection{Geometry of the Grassmann Manifold}\label{sec:grassmann}

We briefly recall the geometric structure of the Grassmann manifold needed for the construction and analysis of the proposed interpolants. Let \(\Gr(r,m)\) denote the set of all \(r\)-dimensional linear subspaces of \(\bR^m\). A point on \(\Gr(r,m)\) is a subspace rather than a particular matrix. It is convenient to represent such a point by an orthonormal basis \(Y\in\bR^{m\times r}\), \(Y^T Y=I_r\), and we write \(\cY=\spanop(Y)\). This representation is non-unique: for any orthogonal matrix \(R\in O(r)\), the matrices \(Y\) and \(YR\) represent the same subspace. Thus \(\Gr(r,m)\) may be viewed as the quotient of the Stiefel manifold \(\St(r,m)=\{Y\in\bR^{m\times r}:Y^TY=I_r\}\) by the right action of \(O(r)\). Equivalently, a Grassmann point may be represented by the orthogonal projector \(P=YY^T\), which is invariant under the change of basis \(Y\mapsto YR\).

For \(\cY=\spanop(Y)\), the tangent space may be identified with the set of matrices
\[
    T_{\cY}\Gr(r,m)=\{\Delta\in\bR^{m\times r}:Y^T\Delta=0\}.
\]
This representation corresponds to infinitesimal changes of the subspace that are orthogonal to the current basis. We use the canonical metric induced by the Frobenius inner product, namely \(\inner{\Delta_1}{\Delta_2}=\tr(\Delta_1^T\Delta_2)\) for \(\Delta_1,\Delta_2\in T_{\cY}\Gr(r,m)\). The corresponding norm is \(\norm{\Delta}_F=\inner{\Delta}{\Delta}^{1/2}\). All interpolation errors below are measured intrinsically on \(\Gr(r,m)\), and therefore do not depend on the particular matrix bases chosen to represent the subspaces.

Given two subspaces \(\cY_0=\spanop(Y_0)\) and \(\cY_1=\spanop(Y_1)\), their relative position is characterized by the principal angles \(\theta_1,\ldots,\theta_r\in[0,\pi/2]\). These angles are obtained from the singular value decomposition
\[
    Y_0^T Y_1=A\cos\Theta B^T,
    \qquad
    \Theta=\diag(\theta_1,\ldots,\theta_r).
\]
The Grassmann geodesic distance is
\[
    d_{\Gr}(\cY_0,\cY_1)=\norm{\Theta}_F
    =\left(\sum_{j=1}^r\theta_j^2\right)^{1/2}.
\]
An equivalent basis-invariant error measure is obtained from the associated projectors \(P_i=Y_iY_i^T\). If \(\theta_j\) are the principal angles between the two subspaces, then
\[
    \frac{1}{\sqrt2}\norm{P_0-P_1}_F
    =\left(\sum_{j=1}^r\sin^2\theta_j\right)^{1/2}.
\]
Hence the projection-matrix error and the intrinsic Grassmann distance are locally equivalent when the subspaces are close.

The logarithm map on \(\Gr(r,m)\) is single-valued only inside a normal neighborhood. In the present setting, this requires the relevant principal angles to be strictly smaller than \(\pi/2\). This condition is the Grassmann analogue of the small-sector condition used in spherical or quaternionic interpolation to avoid non-uniqueness of the shortest geodesic. Throughout the construction, we assume that the sampled subspaces are sufficiently close so that all logarithms and geodesic interpolations are well-defined.

\subsection{Exponential and Logarithm Maps}

The exponential and logarithm maps provide the intrinsic replacement for linear displacement on \(\Gr(r,m)\). Let \(\Delta\in T_{\cY}\Gr(r,m)\), and compute its compact singular value decomposition \(\Delta=U\Theta V^T\), where \(U^TU=V^TV=I\) and \(\Theta\) is diagonal with nonnegative entries. The Grassmann exponential map is
\[
    \Exp_{\cY}(\Delta)=\spanop\left(YV\cos\Theta+U\sin\Theta\right).
\]
Equivalently, one may use \(Z=(YV\cos\Theta+U\sin\Theta)V^T\) as a matrix representative of the resulting subspace, followed by a QR factorization if a numerically orthonormal basis is desired. This formula shows that a geodesic starting from \(\cY=\spanop(Y)\) in the tangent direction \(\Delta=U\Theta V^T\) rotates the basis \(YV\) toward the orthogonal directions \(U\), with angular speeds given by the diagonal entries of \(\Theta\).

Conversely, let \(\cY_0=\spanop(Y_0)\) and \(\cY_1=\spanop(Y_1)\) be sufficiently close so that \(Y_0^TY_1\) is invertible. Define \(A=Y_1(Y_0^TY_1)^{-1}-Y_0\). Since \(Y_0^TA=0\), this matrix lies in the tangent representation at \(\cY_0\). If \(A=U\Sigma V^T\) is its compact singular value decomposition, then the Grassmann logarithm is
\[
    \Log_{\cY_0}(\cY_1)=U\arctan(\Sigma)V^T.
\]
The singular values of \(\Log_{\cY_0}(\cY_1)\) are precisely the principal angles between \(\cY_0\) and \(\cY_1\), and therefore \(\norm{\Log_{\cY_0}(\cY_1)}_F=d_{\Gr}(\cY_0,\cY_1)\). This identity will be used repeatedly in the error analysis, because it allows the interpolation error to be measured either as a Grassmann distance or as the norm of a tangent vector in normal coordinates.

For two nearby points \(\cY_0,\cY_1\in\Gr(r,m)\), the geodesic joining them can now be written as \(\gamma(\tau)=\Exp_{\cY_0}(\tau\Log_{\cY_0}(\cY_1))\), \(0\le \tau\le 1\). This is the basic Grassmann analogue of spherical linear interpolation. In the following sections, we refer to this geodesic interpolation as \(\GLERP\) and use it as the building block for higher-order Grassmann interpolation schemes.

\section{GLERP: Grassmann Linear Interpolation}\label{sec:glerp}

We now introduce the geodesic interpolation that will serve as the basic building block for the higher-order Grassmann interpolants. This construction is the direct analogue of spherical linear interpolation, with the exponential and logarithm maps on \(\Gr(r,m)\) replacing the quaternionic exponential and logarithm. Since a Grassmann point represents a subspace rather than a particular basis, all definitions below are intrinsic and are invariant under the replacement \(Y\mapsto YR\), where \(R\in O(r)\).

\begin{definition}[Grassmann linear interpolation]\label{def:glerp}
Let \(\mathcal{Y}_0,\mathcal{Y}_1\in\Gr(r,m)\) be two Grassmann points such that \(\mathcal{Y}_1\) lies in the injectivity neighborhood of \(\mathcal{Y}_0\). The Grassmann linear interpolation, denoted by \(\GLERP\), is defined by
\[
    \GLERP(\mathcal{Y}_0,\mathcal{Y}_1;\tau)
    =
    \Exp_{\mathcal{Y}_0}\bigl(\tau\Log_{\mathcal{Y}_0}(\mathcal{Y}_1)\bigr),
    \qquad 0\le \tau\le 1.
\]
\end{definition}

The injectivity-neighborhood condition ensures that the logarithm map is uniquely defined and that the interpolation follows the minimizing geodesic. In terms of principal angles, a sufficient local condition is that all principal angles between the two subspaces are strictly smaller than \(\pi/2\). Under this condition, GLERP satisfies the endpoint interpolation conditions \(\GLERP(\mathcal{Y}_0,\mathcal{Y}_1;0)=\mathcal{Y}_0\) and \(\GLERP(\mathcal{Y}_0,\mathcal{Y}_1;1)=\mathcal{Y}_1\). Moreover, it moves with constant speed along the geodesic,
\[
    d_{\Gr}\bigl(\GLERP(\mathcal{Y}_0,\mathcal{Y}_1;s),
    \GLERP(\mathcal{Y}_0,\mathcal{Y}_1;t)\bigr)
    =
    |s-t|\,d_{\Gr}(\mathcal{Y}_0,\mathcal{Y}_1),
\]
whenever the same minimizing geodesic remains valid.

To obtain a computable formula, let \(Y_0\) and \(Y_1\) be orthonormal representatives of \(\mathcal{Y}_0\) and \(\mathcal{Y}_1\), respectively. If \(\Log_{\mathcal{Y}_0}(\mathcal{Y}_1)=U\Theta V^T\) is the compact singular value decomposition of the tangent vector, then a representative of the interpolated subspace is
\[
    Z(\tau)
    =
    \bigl(Y_0V\cos(\tau\Theta)+U\sin(\tau\Theta)\bigr)V^T.
\]
Thus \(\GLERP(\mathcal{Y}_0,\mathcal{Y}_1;\tau)=\spanop(Z(\tau))\). In exact arithmetic, \(Z(\tau)\) has orthonormal columns; in floating point computations, a final QR orthonormalization may be applied.

\begin{algorithm}
\caption{GLERP by exponential and logarithm maps}
\KwIn{Matrices \(Y_0,Y_1\in\bR^{m\times r}\) representing two nearby subspaces; parameter \(\tau\in[0,1]\)}
\KwOut{An orthonormal basis \(Z\) for \(\GLERP(\spanop(Y_0),\spanop(Y_1);\tau)\)}
Orthonormalize \(Y_0\) and \(Y_1\), if necessary\;
Compute \(M=Y_0^TY_1\)\;
Check that \(M\) is nonsingular; otherwise the logarithm is not uniquely determined by this local formula\;
Set \(A=Y_1M^{-1}-Y_0\)\;
Compute the compact SVD \(A=U\Sigma V^T\)\;
Set \(\Theta=\arctan(\Sigma)\)\;
Set \(Z=(Y_0V\cos(\tau\Theta)+U\sin(\tau\Theta))V^T\)\;
Orthonormalize \(Z\) by QR if necessary\;
\end{algorithm}

\subsection{Second-Order Accuracy of GLERP}

We next show that GLERP is second-order accurate for smooth Grassmann-valued curves. The result is local: the sampled points must remain in a normal neighborhood so that the logarithm map is single-valued. This is the same geometric restriction that will later appear in the higher-order construction, where the recursive interpolants may also generate intermediate subspaces.

\begin{prop}[Second-order accuracy of GLERP]\label{prop:glerp_second_order}
Let \(\gamma:[0,T]\to\Gr(r,m)\) be a \(C^2\) curve. Let \(t_i=ih\), \(\mathcal{Y}_i=\gamma(t_i)\), and define, for \(t=t_i+\tau h\) with \(0\le \tau\le 1\),
\[
    \widehat{\gamma}_h(t)
    =
    \GLERP(\mathcal{Y}_i,\mathcal{Y}_{i+1};\tau).
\]
Assume that \(h\) is sufficiently small so that \(\gamma([t_i,t_{i+1}])\), \(\mathcal{Y}_i\), and \(\mathcal{Y}_{i+1}\) lie in a common normal coordinate neighborhood. Then there exists a constant \(C\), independent of \(h\) and \(\tau\), such that
\[
    d_{\Gr}\bigl(\gamma(t_i+\tau h),
    \widehat{\gamma}_h(t_i+\tau h)\bigr)
    \le
    Ch^2,
    \qquad 0\le \tau\le 1.
\]
\end{prop}

\begin{proof}
Fix an interval \([t_i,t_{i+1}]\) and set \(\mathcal{P}=\gamma(t_i)\). In normal coordinates centered at \(\mathcal{P}\), define \(\xi(s)=\Log_{\mathcal{P}}\gamma(t_i+s)\) for \(0\le s\le h\). Then \(\xi\) is a \(C^2\) curve in the vector space \(T_{\mathcal{P}}\Gr(r,m)\), with \(\xi(0)=0\) and \(\xi(h)=\Log_{\mathcal{P}}\gamma(t_i+h)\). By the definition of GLERP, the interpolant satisfies
\[
    \widehat{\gamma}_h(t_i+s)=\Exp_{\mathcal{P}}\ell(s),
    \qquad
    \ell(s)=\frac{s}{h}\xi(h).
\]
Thus, in normal coordinates, GLERP is simply the straight-line interpolant of the tangent-space curve \(\xi\) between \(s=0\) and \(s=h\).

The standard interpolation estimate for a \(C^2\) vector-valued function gives
\[
    \norm{\xi(s)-\ell(s)}
    \le
    \frac{s(h-s)}{2}
    \max_{0\le \eta\le h}\norm{\xi''(\eta)}
    \le
    Ch^2,
    \qquad 0\le s\le h.
\]
Since the curve remains in a compact subset of a normal coordinate neighborhood, the exponential map is smooth and locally Lipschitz on the relevant tangent-space ball. Hence
\[
    d_{\Gr}\bigl(\Exp_{\mathcal{P}}\xi(s),
    \Exp_{\mathcal{P}}\ell(s)\bigr)
    \le
    C\norm{\xi(s)-\ell(s)}
    \le
    Ch^2.
\]
Using \(\Exp_{\mathcal{P}}\xi(s)=\gamma(t_i+s)\) and \(\Exp_{\mathcal{P}}\ell(s)=\widehat{\gamma}_h(t_i+s)\) proves the desired estimate.
\end{proof}

\begin{remark}[Leading error term]
If \(\gamma\) is \(C^3\), the preceding argument also gives the local expansion
\[
    \xi(\tau h)-\ell(\tau h)
    =
    -\frac{1}{2}\tau(1-\tau)h^2\xi''(0)
    +
    \mathcal{O}(h^3).
\]
Thus the leading GLERP error is governed by the second derivative of the exact curve in normal coordinates. In particular, if \(\gamma\) is itself a Grassmann geodesic, then \(\xi(s)=sV\) for some fixed tangent vector \(V\), and GLERP is exact.
\end{remark}

\begin{remark}[Projection-matrix error]
If \(Y(t)\) and \(\widehat{Y}_h(t)\) are orthonormal representatives of \(\gamma(t)\) and \(\widehat{\gamma}_h(t)\), respectively, the same accuracy may be measured by the associated projectors \(P(t)=Y(t)Y(t)^T\) and \(\widehat{P}_h(t)=\widehat{Y}_h(t)\widehat{Y}_h(t)^T\). Since the projection-matrix distance and the Grassmann geodesic distance are locally equivalent, the estimate above implies
\[
    \norm{P(t)-\widehat{P}_h(t)}_F
    =
    \mathcal{O}(h^2)
\]
uniformly on each interpolation interval.
\end{remark}

\section{Recursive GIDER Interpolation}\label{sec:gider}

We now construct higher-order Grassmann interpolants by replacing the linear interpolation step in Neville's algorithm by \(\GLERP\).  The resulting method is recursive, intrinsic, and depends only on geodesic interpolation on \(\Gr(r,m)\).  Since the recursion may evaluate geodesics outside the interval between their endpoints, we use the same formula \(\GLERP(\mathcal{A},\mathcal{B};\lambda)=\Exp_{\mathcal{A}}\bigl(\lambda\Log_{\mathcal{A}}(\mathcal{B})\bigr)\) also for real parameters \(\lambda\notin[0,1]\), whenever the logarithm and the corresponding geodesic extrapolation remain well-defined.

Suppose that \(n+1\) Grassmann data points \(\mathcal{Y}_0,\mathcal{Y}_1,\ldots,\mathcal{Y}_n\in\Gr(r,m)\) are sampled at equally spaced parameter values \(t_j=t_0+jh\), \(j=0,1,\ldots,n\).

\begin{definition}[Recursive \(\GIDER_n\)]\label{def:gider}
Set
\[
    \mathcal{I}_j^0(t)=\mathcal{Y}_j,
    \qquad j=0,1,\ldots,n.
\]
For \(k=1,2,\ldots,n\), define recursively
\[
    \mathcal{I}_j^k(t)
    =
    \GLERP\left(
        \mathcal{I}_j^{k-1}(t),
        \mathcal{I}_{j+1}^{k-1}(t);
        \lambda_j^k(t)
    \right),
    \qquad
    j=0,1,\ldots,n-k,
\]
where
\[
    \lambda_j^k(t)=\frac{t-t_j}{t_{j+k}-t_j}.
\]
The \(n\)-th order Grassmann interpolation of the data is defined by
\[
    \GIDER_n(\mathcal{Y}_0,\ldots,\mathcal{Y}_n;t)=\mathcal{I}_0^n(t),
    \qquad t\in[t_0,t_n].
\]
\end{definition}

This is the Grassmann analogue of Neville interpolation.  In Euclidean space, replacing \(\GLERP\) by ordinary linear interpolation recovers the classical Neville algorithm and hence the polynomial interpolant through the data.  On the Grassmann manifold, the same recursive structure is retained, but each affine combination is replaced by a geodesic interpolation or extrapolation.

\subsection{The Case of \(\GIDER_2\)}

For three data points \(\mathcal{Y}_0,\mathcal{Y}_1,\mathcal{Y}_2\) sampled at \(t_0,t_1,t_2\), with \(t_1=(t_0+t_2)/2\), define the normalized parameter \(\tau=(t-t_0)/(t_2-t_0)\).  Then \(t=t_0+2h\tau\), and the recursive definition gives
\[
\begin{split}
    \GIDER_2(\mathcal{Y}_0,\mathcal{Y}_1,\mathcal{Y}_2;\tau)
    =
    \GLERP\Big(&
        \GLERP(\mathcal{Y}_0,\mathcal{Y}_1;2\tau),\\
        &
        \GLERP(\mathcal{Y}_1,\mathcal{Y}_2;2\tau-1);
        \tau
    \Big).
\end{split}
\]
Thus \(\GIDER_2\) is a \(\GLERP\) of two \(\GLERP\) interpolants.  If the manifold is replaced by a Euclidean vector space, this formula reduces to the quadratic Lagrange interpolant through the nodes \(0\), \(1/2\), and \(1\), namely
\[
    (1-\tau)(1-2\tau)y_0
    +
    4\tau(1-\tau)y_1
    +
    \tau(2\tau-1)y_2.
\]
This interpretation is different from the control-point version of \(\SIDER_2\) in \cite{FongLeung2023}; the present construction is instead a direct manifold analogue of quadratic Lagrange interpolation.

\begin{remark}[Geodesic extrapolation and bounded parameters]
In the formula for \(\GIDER_2\), the inner \(\GLERP\) parameters are \(2\tau\) and \(2\tau-1\).  Hence, for \(\tau\in[0,1]\), one of the inner geodesic evaluations may be an extrapolation outside the interval \([0,1]\).  More generally, \(\GIDER_n\) contains intermediate \(\GLERP\) evaluations with parameters outside \([0,1]\).  For a fixed order \(n\), however, these parameters are bounded.  More precisely, define
\[
    \Lambda_n
    =
    \max_{1\le k\le n}\max_{0\le j\le n-k}
    \sup_{t\in[t_0,t_n]}
    \left|\lambda_j^k(t)\right|.
\]
For a fixed set of interpolation nodes, \(\Lambda_n<\infty\).  In particular, for equally spaced nodes \(t_j=t_0+jh\), one has \(\lambda_j^k(t_0+h\sigma)=(\sigma-j)/k\), \(0\le \sigma\le n\), and hence \(\Lambda_n\le n\).  The following proposition gives a quantitative sufficient condition ensuring that the geodesic extrapolations remain admissible.
\end{remark}

\begin{prop}[Quantitative admissibility condition]
\label{prop:quantitative_admissibility}
Fix the interpolation order \(n\) and a target interval \([t_0,t_n]\).  Let \(q\in\Gr(r,m)\) be a reference subspace, and suppose that the input data satisfy
\[
    \max_{0\le j\le n} d_{\Gr}(q,\mathcal{Y}_j) \le A h
\]
for some constant \(A\) independent of \(h\).  Let \(\Lambda_n\) be defined as above, and set
\[
    C_0=A,
    \qquad
    C_k=(1+2\Lambda_n)C_{k-1},
    \qquad k=1,\ldots,n.
\]
If
\[
    2C_{n-1}h < \frac{\pi}{2},
\]
then every pair of intermediate subspaces appearing as the two arguments of a \(\GLERP\) step is connected by a unique minimizing geodesic.  More precisely, every level-\(k\) intermediate subspace satisfies
\[
    d_{\Gr}\bigl(q,\mathcal{I}_j^k(t)\bigr)
    \le C_k h,
    \qquad
    k=0,1,\ldots,n,
\]
and every pair used in a level-\(k\) recursive step satisfies
\[
    \theta_{\max}\bigl(\mathcal{I}_j^{k-1}(t),
    \mathcal{I}_{j+1}^{k-1}(t)\bigr)
    \le
    d_{\Gr}\bigl(\mathcal{I}_j^{k-1}(t),
    \mathcal{I}_{j+1}^{k-1}(t)\bigr)
    \le 2C_{k-1}h
    \le 2C_{n-1}h,
\]
where \(\theta_{\max}\) denotes the largest principal angle.  Consequently, if \(\eta<\pi/2\) and \(2C_{n-1}h\le \eta\), then all principal angles required by the recursive construction are at least \(\pi/2-\eta\) away from the cut-locus threshold.
\end{prop}

\begin{proof}
The proof is by induction over the recursion level.  At level \(k=0\), the desired bound is precisely the assumed estimate on the input data.  Suppose that all level-\((k-1)\) intermediate subspaces satisfy
\[
    d_{\Gr}\bigl(q,\mathcal{I}_j^{k-1}(t)\bigr)
    \le C_{k-1}h.
\]
Then, by the triangle inequality,
\[
    d_{\Gr}\bigl(\mathcal{I}_j^{k-1}(t),
    \mathcal{I}_{j+1}^{k-1}(t)\bigr)
    \le 2C_{k-1}h.
\]
If \(2C_{n-1}h<\pi/2\), then this distance is smaller than \(\pi/2\) for every \(k\le n\).  Since the largest principal angle is bounded above by the Grassmann distance, the logarithm map used in this \(\GLERP\) step is uniquely defined.

Let
\[
    Z=\GLERP\left(
        \mathcal{I}_j^{k-1}(t),
        \mathcal{I}_{j+1}^{k-1}(t);
        \lambda_j^k(t)
    \right).
\]
The curve defining this \(\GLERP\) step has length at most
\[
    \left|\lambda_j^k(t)\right|
    d_{\Gr}\bigl(\mathcal{I}_j^{k-1}(t),
    \mathcal{I}_{j+1}^{k-1}(t)\bigr).
\]
Therefore,
\[
\begin{split}
    d_{\Gr}(q,Z)
    &\le
    d_{\Gr}\bigl(q,\mathcal{I}_j^{k-1}(t)\bigr)
    +
    \left|\lambda_j^k(t)\right|
    d_{\Gr}\bigl(\mathcal{I}_j^{k-1}(t),
    \mathcal{I}_{j+1}^{k-1}(t)\bigr)  \\
    &\le C_{k-1}h+2\Lambda_n C_{k-1}h
    = C_k h.
\end{split}
\]
This proves the induction step and hence the proposition.
\end{proof}

The hypothesis in Proposition~\ref{prop:quantitative_admissibility} is satisfied automatically for sufficiently dense samples from a smooth curve.  For example, if \(\mathcal{Y}_j=\gamma(t_j)\), the nodes are equally spaced, and \(V=\sup_{t\in[t_0,t_n]}\|\dot\gamma(t)\|\), then choosing \(q=\gamma(t_0)\) gives
\[
    d_{\Gr}(q,\mathcal{Y}_j)
    \le V(t_j-t_0)
    \le nVh.
\]
Thus one may take \(A=nV\).  For fixed \(n\), the constants \(\Lambda_n\) and \(C_k\) are independent of \(h\), and the condition \(2C_{n-1}h<\pi/2\) holds for all sufficiently small \(h\).  This makes explicit that the relevant principal angles do not approach \(\pi/2\) under refinement; instead, they remain \(\mathcal{O}(h)\) and therefore stay uniformly inside the injectivity neighborhood.

\begin{algorithm}
\caption{Recursive \(\GIDER_n\)}
\KwIn{Grassmann data \(\mathcal{Y}_0,\ldots,\mathcal{Y}_n\) at \(t_0,\ldots,t_n\), target time \(t\)}
\KwOut{\(\widehat{\mathcal{Y}}=\GIDER_n(\mathcal{Y}_0,\ldots,\mathcal{Y}_n;t)\)}
Set \(\mathcal{I}_j^0=\mathcal{Y}_j\), \(j=0,\ldots,n\)\;
\For{\(k=1,\ldots,n\)}{
    \For{\(j=0,\ldots,n-k\)}{
        Set \(\lambda=(t-t_j)/(t_{j+k}-t_j)\)\;
        Set \(\mathcal{I}_j^k=\GLERP(\mathcal{I}_j^{k-1},\mathcal{I}_{j+1}^{k-1};\lambda)\)\;
    }
}
Return \(\widehat{\mathcal{Y}}=\mathcal{I}_0^n\)\;
\end{algorithm}

\subsection{Interpolation Property}

\begin{prop}[Interpolation property]\label{prop:interpolation_property}
Assume that all \(\GLERP\) evaluations in the recursive construction are well-defined.  Then the \(\GIDER_n\) interpolant satisfies
\[
    \GIDER_n(\mathcal{Y}_0,\ldots,\mathcal{Y}_n;t_j)=\mathcal{Y}_j,
    \qquad j=0,1,\ldots,n.
\]
\end{prop}

\begin{proof}
We prove a slightly more general statement.  For each \(k=0,1,\ldots,n\), the interpolant \(\mathcal{I}_j^k(t)\) interpolates the local data \(\mathcal{Y}_j,\mathcal{Y}_{j+1},\ldots,\mathcal{Y}_{j+k}\), namely \(\mathcal{I}_j^k(t_\ell)=\mathcal{Y}_\ell\) for \(\ell=j,j+1,\ldots,j+k\).  The claim is immediate for \(k=0\).  Assume that it holds for order \(k-1\).  At the left endpoint \(t=t_j\), we have \(\lambda_j^k(t_j)=0\), and therefore \(\mathcal{I}_j^k(t_j)=\mathcal{I}_j^{k-1}(t_j)=\mathcal{Y}_j\).  At the right endpoint \(t=t_{j+k}\), we have \(\lambda_j^k(t_{j+k})=1\), and therefore \(\mathcal{I}_j^k(t_{j+k})=\mathcal{I}_{j+1}^{k-1}(t_{j+k})=\mathcal{Y}_{j+k}\).  For an interior node \(t_\ell\), \(j<\ell<j+k\), the induction hypothesis gives \(\mathcal{I}_j^{k-1}(t_\ell)=\mathcal{Y}_\ell\) and \(\mathcal{I}_{j+1}^{k-1}(t_\ell)=\mathcal{Y}_\ell\).  Hence the recursive step gives
\[
    \mathcal{I}_j^k(t_\ell)
    =
    \GLERP(\mathcal{Y}_\ell,\mathcal{Y}_\ell;\lambda_j^k(t_\ell))
    =
    \mathcal{Y}_\ell.
\]
This proves the local interpolation property for all \(k\).  Taking \(j=0\) and \(k=n\) gives the stated result.
\end{proof}

\begin{remark}[Injectivity radius and admissible data]\label{rem:injectivity_gider}
The preceding formulas are local in nature.  The logarithm map \(\Log_{\mathcal{Y}_0}(\mathcal{Y}_1)\) is uniquely defined only when \(\mathcal{Y}_1\) lies inside the injectivity neighborhood of \(\mathcal{Y}_0\).  For the Grassmann manifold with the canonical metric, a sufficient condition is that the largest principal angle between the two subspaces be strictly smaller than \(\pi/2\).  Proposition~\ref{prop:quantitative_admissibility} gives a concrete small-data-spacing condition for \(\GIDER_n\): for fixed order \(n\), if the data are \(\mathcal{O}(h)\)-close to a reference subspace and \(h\) is sufficiently small relative to the bounded extrapolation constant \(\Lambda_n\), then all intermediate \(\GLERP\) steps remain uniformly away from the cut locus.  This is the admissibility condition assumed in the local analysis below.
\end{remark}

\section{Accuracy of \texorpdfstring{\(\GIDER_n\)}{GIDERn}}\label{sec:accuracy}

We next establish a local accuracy result for the recursive \(\GIDER_n\) construction.  The argument follows the classical interpolation-error proof, but it must be formulated in a local coordinate chart because the interpolant is Grassmann-valued.  The main idea is simple: in a sufficiently small normal neighborhood, the recursive geodesic construction is a smooth perturbation of Neville interpolation, the interpolant agrees with the exact curve at the \(n+1\) interpolation nodes, and the coordinate error therefore contains the product of the nodal factors.

\begin{assump}[Local normal-neighborhood assumption]\label{assump:normal}
Let \(\gamma:[a,b]\to\Gr(r,m)\) be a smooth curve.  For the interpolation interval \([t_0,t_n]\), assume that \(h\) is sufficiently small so that the exact curve, the data points, and all intermediate \(\GLERP\) evaluations in the \(\GIDER_n\) recursion remain in a compact subset \(K\) of a common normal coordinate neighborhood \(U\subset\Gr(r,m)\).  In particular, all logarithm maps used by \(\GLERP\) are uniquely defined, and all geodesic interpolation and extrapolation steps in the recursion are well-defined.
\end{assump}

For the Grassmann manifold with the canonical metric, a practical sufficient condition is that all relevant principal angles remain bounded away from \(\pi/2\).  Proposition~\ref{prop:quantitative_admissibility} provides one explicit way to verify this condition: for fixed interpolation order \(n\), the recursive extrapolation parameters are bounded, and sufficiently dense samples from a smooth curve keep all intermediate principal angles of size \(\mathcal{O}(h)\).

We first record two elementary facts that will be used in the proof.  The first states that the recursive \(\GIDER_n\) curve has uniformly bounded derivatives in a local coordinate chart.  The second is the standard interpolation remainder estimate written in a form suitable for the coordinate error.

\begin{lemma}[Local coordinate stability]\label{lemma:local_stability}
Let Assumption \ref{assump:normal} hold, and let \(\phi:U\to\bR^d\), \(d=r(m-r)\), be a smooth coordinate chart on the normal neighborhood \(U\), with \(K\Subset U\).  Define
\[
    F_h(t)
    =
    \phi\left(\GIDER_n(\gamma(t_0),\ldots,\gamma(t_n);t)\right),
    \qquad t\in[t_0,t_n].
\]
If \(\gamma\in C^{n+1}\), then \(F_h\in C^{n+1}([t_0,t_n];\bR^d)\), and there exists a constant \(C\), independent of sufficiently small \(h\), such that
\[
    \max_{t\in[t_0,t_n]}
    \left\|
        \frac{d^{\,n+1}}{dt^{\,n+1}}F_h(t)
    \right\|
    \le C.
\]
\end{lemma}

\begin{proof}
The map \(\GLERP\) is a composition of the Grassmann logarithm map, scalar multiplication in a tangent space, and the Grassmann exponential map.  These maps are smooth on \(K\) because \(K\) lies inside a normal neighborhood and stays away from the cut locus.  Therefore every finite recursive composition appearing in \(\GIDER_n\) is smooth in \(t\).

It remains to justify that the derivatives do not blow up as \(h\to0\).  Write \(t=t_0+h\sigma\), where \(\sigma\in[0,n]\), and define the rescaled coordinate curve
\[
    \widetilde F_h(\sigma)
    =
    F_h(t_0+h\sigma).
\]
We shall show that
\[
    \frac{d^{\,q}}{d\sigma^{q}}\widetilde F_h(\sigma)
    =
    \mathcal{O}(h^q),
    \qquad q=1,\ldots,n+1,
\]
uniformly for \(\sigma\in[0,n]\).  This implies the desired estimate, because
\[
    \frac{d^{\,q}}{dt^q}F_h(t)
    =
    h^{-q}
    \frac{d^{\,q}}{d\sigma^q}\widetilde F_h(\sigma).
\]

The claim follows from the local expansion of geodesic interpolation.  Fix a point \(p\in K\), and use normal coordinates \(x=\phi(\mathcal{X})\) centered in a slightly larger coordinate neighborhood.  In these coordinates, the map
\[
    G(x,y;\lambda)
    =
    \phi\left(
        \GLERP(\phi^{-1}(x),\phi^{-1}(y);\lambda)
    \right)
\]
is smooth for \(x,y\in\phi(K)\) and for all values of \(\lambda\) generated by the fixed-order recursion.  Since \(G(x,x;\lambda)=x\) and the geodesic equation has smooth coefficients in normal coordinates, the small-displacement expansion has the form
\[
    G(x,y;\lambda)
    =
    (1-\lambda)x+\lambda y
    +
    \mathcal{R}(x,y;\lambda),
\]
where, for \(x\) and \(y\) sufficiently close,
\[
    \left\|\partial_\lambda^q \mathcal{R}(x,y;\lambda)\right\|
    \le
    C_q\|x-y\|^2,
    \qquad q=0,1,\ldots,n+1.
\]
The constants \(C_q\) are uniform because \(K\) is compact and the admissible \(\lambda\)-values are bounded for fixed \(n\).

The sampled data satisfy
\[
    \phi(\gamma(t_0+jh))
    =
    \phi(\gamma(t_0))
    +
    \sum_{\ell=1}^{n+1}
    \frac{(jh)^\ell}{\ell!}
    \frac{d^\ell}{dt^\ell}(\phi\circ\gamma)(t_0)
    +
    \mathcal{O}(h^{n+2}),
\]
uniformly for \(j=0,\ldots,n\).  In particular, neighboring sampled data differ by \(\mathcal{O}(h)\).  Applying the preceding expansion of \(G\) recursively, one obtains by induction over the recursion level \(k\) that each rescaled intermediate curve
\[
    \widetilde F_{j,k,h}(\sigma)
    =
    \phi\left(\mathcal{I}_j^k(t_0+h\sigma)\right)
\]
admits an expansion of the form
\[
    \widetilde F_{j,k,h}(\sigma)
    =
    \sum_{\ell=0}^{n+1}
    h^\ell P_{j,k,\ell}(\sigma)
    +
    \mathcal{O}(h^{n+2}),
\]
where the coefficients \(P_{j,k,\ell}\) are smooth functions on the fixed interval \([0,n]\), and, for \(\ell\le k\), they agree with the coefficients generated by the ordinary Neville interpolation of the Taylor coefficients of \(\phi\circ\gamma\).  In particular, the \(\sigma\)-derivatives of order \(q\) satisfy
\[
    \frac{d^q}{d\sigma^q}\widetilde F_{j,k,h}(\sigma)=\mathcal{O}(h^q),
    \qquad q=1,\ldots,n+1,
\]
uniformly for all \(j,k\) appearing in the recursion.  The induction uses that each normalized parameter \(\lambda_j^k(t_0+h\sigma)=(\sigma-j)/k\) is independent of \(h\), while the displacement between the two lower-level inputs to each \(\GLERP\) step is \(\mathcal{O}(h)\).  The remainder term \(\mathcal{R}\) is therefore at least quadratic in a quantity of size \(\mathcal{O}(h)\), and its derivatives with respect to \(\sigma\) satisfy the same scaling after the chain rule is applied.  Taking \(j=0\) and \(k=n\) gives the asserted bounds for \(\widetilde F_h\), and hence for \(F_h\).
\end{proof}

\begin{lemma}[Root estimate]\label{lemma:root_estimate}
Let \(E:[t_0,t_n]\to\bR^d\) be a \(C^{n+1}\) function such that \(E(t_j)=0\) for \(j=0,1,\ldots,n\), where \(t_j=t_0+jh\).  If
\[
    \max_{t\in[t_0,t_n]}\left\|E^{(n+1)}(t)\right\|\le M,
\]
then there exists a constant \(C_n\), depending only on \(n\), such that
\[
    \|E(t)\|
    \le
    C_n M h^{n+1},
    \qquad t\in[t_0,t_n].
\]
\end{lemma}

\begin{proof}
It is enough to prove the estimate componentwise.  Let \(e\) be one scalar component of \(E\).  For any fixed \(t\in[t_0,t_n]\), consider the \(n+1\) interpolation nodes \(t_0,\ldots,t_n\) together with the point \(t\).  Since \(e(t_j)=0\), the divided-difference formula gives
\[
    e(t)
    =
    e[t_0,\ldots,t_n,t]
    \prod_{j=0}^n (t-t_j),
\]
where \(e[t_0,\ldots,t_n,t]\) denotes the divided difference of order \(n+1\), with repeated nodes interpreted in the usual limiting sense if \(t\) coincides with one of the nodes.  By the mean-value theorem for divided differences, there exists \(\xi\in[t_0,t_n]\) such that
\[
    e[t_0,\ldots,t_n,t]
    =
    \frac{e^{(n+1)}(\xi)}{(n+1)!}.
\]
Hence
\[
    |e(t)|
    \le
    \frac{M}{(n+1)!}
    \prod_{j=0}^n |t-t_j|
    \le
    C_n M h^{n+1},
\]
because \(t\in[t_0,t_n]\) and all nodal distances are bounded by \(nh\).  Summing over the finitely many coordinate components gives the vector estimate.
\end{proof}

\begin{thm}[Accuracy of \(\GIDER_n\)]\label{thm:gider_accuracy}
Let \(\gamma:[a,b]\to\Gr(r,m)\) be a \(C^{n+1}\) curve, and suppose that
\[
    \mathcal{Y}_j=\gamma(t_j),
    \qquad
    t_j=t_0+jh,
    \qquad j=0,1,\ldots,n.
\]
Let
\[
    \widehat{\gamma}_h(t)
    =
    \GIDER_n(\mathcal{Y}_0,\ldots,\mathcal{Y}_n;t),
    \qquad t\in[t_0,t_n].
\]
Under Assumption \ref{assump:normal}, there exists a constant \(C\), independent of sufficiently small \(h\), such that
\[
    d_{\Gr}\bigl(\gamma(t),\widehat{\gamma}_h(t)\bigr)
    \le
    C h^{n+1},
    \qquad t\in[t_0,t_n].
\]
Therefore \(\GIDER_n\) is locally \((n+1)\)-st order accurate.
\end{thm}

\begin{proof}
Choose a smooth coordinate chart \(\phi:U\to\bR^d\), \(d=r(m-r)\), whose domain contains the compact set \(K\) from Assumption \ref{assump:normal}.  Define the coordinate error
\[
    E(t)
    =
    \phi(\widehat{\gamma}_h(t))
    -
    \phi(\gamma(t)).
\]
By the interpolation property of \(\GIDER_n\), the interpolant agrees with the exact curve at every data node.  Hence
\[
    E(t_j)=0,
    \qquad j=0,1,\ldots,n.
\]

Since \(\gamma\in C^{n+1}\), the coordinate curve \(\phi\circ\gamma\) has a bounded \((n+1)\)-st derivative on \([t_0,t_n]\).  By Lemma \ref{lemma:local_stability}, the same is true for \(t\mapsto\phi(\widehat{\gamma}_h(t))\), with a bound independent of sufficiently small \(h\).  Therefore there exists \(M>0\), independent of \(h\), such that
\[
    \max_{t\in[t_0,t_n]}
    \left\|E^{(n+1)}(t)\right\|
    \le M.
\]
Applying Lemma \ref{lemma:root_estimate} gives
\[
    \|E(t)\|
    \le
    C h^{n+1},
    \qquad t\in[t_0,t_n].
\]

It remains to convert this coordinate estimate into a Grassmann distance estimate.  Since \(K\Subset U\) is compact and \(\phi\) is a smooth coordinate chart, the Riemannian distance and the Euclidean coordinate distance are locally equivalent on \(K\).  Thus there exists a constant \(C_K>0\) such that, for all \(\mathcal{X},\mathcal{Y}\in K\),
\[
    d_{\Gr}(\mathcal{X},\mathcal{Y})
    \le
    C_K
    \|\phi(\mathcal{X})-\phi(\mathcal{Y})\|.
\]
Taking \(\mathcal{X}=\gamma(t)\) and \(\mathcal{Y}=\widehat{\gamma}_h(t)\), we obtain
\[
    d_{\Gr}\bigl(\gamma(t),\widehat{\gamma}_h(t)\bigr)
    \le
    C_K\|E(t)\|
    \le
    C h^{n+1}.
\]
This proves the theorem.
\end{proof}

\begin{remark}[Normal-coordinate interpretation]\label{rem:normal_coordinate_interpretation}
The proof can also be viewed in normal coordinates.  Fix a point \(q\) in the common normal neighborhood and write nearby Grassmann points as \(\Exp_q u\).  In these coordinates, geodesic interpolation satisfies
\[
    \Log_q\left(
        \GLERP(\Exp_q u,\Exp_q v;\lambda)
    \right)
    =
    (1-\lambda)u+\lambda v+\mathcal{R}(u,v;\lambda),
\]
where the curvature correction satisfies \(\mathcal{R}(u,v;\lambda)=\mathcal{O}((\|u\|+\|v\|)^2)\), uniformly for bounded \(\lambda\).  Thus \(\GLERP\) is a smooth perturbation of ordinary linear interpolation in local coordinates.  The recursive \(\GIDER_n\) construction is therefore a curved-space analogue of Neville interpolation.  The curvature terms remain controlled under Assumption \ref{assump:normal}; together with the interpolation property at the \(n+1\) nodes, this gives the same local order \(n+1\) as the Euclidean interpolant.
\end{remark}

\begin{remark}[Projection-matrix form of the estimate]\label{rem:projection_accuracy}
If \(Y(t)\) and \(\widehat Y_h(t)\) are orthonormal representatives of \(\gamma(t)\) and \(\widehat{\gamma}_h(t)\), respectively, then the same accuracy may be expressed using the associated projectors \(P(t)=Y(t)Y(t)^T\) and \(\widehat P_h(t)=\widehat Y_h(t)\widehat Y_h(t)^T\).  Since the projection-matrix distance and the Grassmann geodesic distance are locally equivalent for nearby subspaces, the theorem also implies
\[
    \|P(t)-\widehat P_h(t)\|_F
    =
    \mathcal{O}(h^{n+1}),
    \qquad t\in[t_0,t_n].
\]
This form is often convenient in numerical experiments because it is independent of the choice of orthonormal bases.
\end{remark}

\section{Measuring Error for Matrix-Valued Data}\label{sec:error}

The Grassmann formulation is appropriate when each matrix represents a subspace rather than a particular basis.  Thus, if \(Y,\widehat Y\in\bR^{m\times r}\) have orthonormal columns, the entrywise error \(\norm{Y-\widehat Y}_F\) is generally not a meaningful Grassmann error.  Indeed, \(Y\) and \(YR\) represent the same point on \(\Gr(r,m)\) for every \(R\in O(r)\).  Error measures should therefore be invariant under the right action \(Y\mapsto YR\).

\subsection{Intrinsic Grassmann Distance}

The most natural error measure is the Grassmann geodesic distance.  Given two orthonormal representatives \(Y\) and \(\widehat Y\), let the singular values of \(Y^T\widehat Y\) be \(\sigma_j=\cos\theta_j\), where \(0\le \theta_j\le \pi/2\) are the principal angles between the two subspaces.  The intrinsic error is
\[
    d_{\Gr}(Y,\widehat Y)
    =
    \left(\sum_{j=1}^r\theta_j^2\right)^{1/2}.
\]
Here and below, \(d_{\Gr}(Y,\widehat Y)\) is a shorthand for \(d_{\Gr}(\spanop(Y),\spanop(\widehat Y))\).  For a numerical convergence study, one may use the maximum error
\[
    E_\infty(h)
    =
    \max_{t\in I} d_{\Gr}\bigl(Y(t),\widehat Y_h(t)\bigr),
\]
or an integrated error such as
\[
    E_2(h)
    =
    \left(
        \int_I d_{\Gr}\bigl(Y(t),\widehat Y_h(t)\bigr)^2\,dt
    \right)^{1/2},
\]
where \(I\) is the interval over which the reconstruction is assessed.  The observed convergence rate is then
\[
    \rho_h
    =
    \log_2\left(\frac{E(h)}{E(h/2)}\right).
\]
Theorem \ref{thm:gider_accuracy} predicts \(\rho_h\to n+1\) for \(\GIDER_n\), provided that the exact curve is sufficiently smooth and that the data and all intermediate geodesic constructions remain in a common normal neighborhood.

\subsection{Projection-Matrix Error}

A convenient basis-invariant alternative is to represent each subspace by its orthogonal projector.  For orthonormal representatives \(Y\) and \(\widehat Y\), define
\[
    P=YY^T,
    \qquad
    \widehat P=\widehat Y\widehat Y^T.
\]
The projection-matrix error
\[
    e_P(Y,\widehat Y)
    =
    \norm{P-\widehat P}_F
\]
is independent of the choice of orthonormal bases.  Its relation to the principal angles is
\[
    \frac{1}{\sqrt{2}}\norm{P-\widehat P}_F
    =
    \left(\sum_{j=1}^r\sin^2\theta_j\right)^{1/2}.
\]
For small errors, \(\sin\theta_j\sim\theta_j\), and hence the projection-matrix error and the Grassmann geodesic distance are locally equivalent.  Consequently, the intrinsic estimate
\[
    d_{\Gr}\bigl(Y(t),\widehat Y_h(t)\bigr)
    =
    \mathcal{O}(h^{n+1})
\]
is equivalent, in a sufficiently small neighborhood, to
\[
    \norm{
        Y(t)Y(t)^T
        -
        \widehat Y_h(t)\widehat Y_h(t)^T
    }_F
    =
    \mathcal{O}(h^{n+1}).
\]
The projection form is often convenient in computations because it avoids any alignment of the two bases.

\subsection{Aligned Basis Error}

If one wants to compare matrix representatives directly, the appropriate invariant version is the orthogonal Procrustes error
\[
    e_{\mathrm{align}}(Y,\widehat Y)
    =
    \min_{R\in O(r)}
    \norm{Y-\widehat YR}_F.
\]
This minimization removes the arbitrary choice of basis inside the reconstructed subspace.  If \(Y^T\widehat Y=A\Sigma B^T\) is a singular value decomposition, then an optimizer is \(R=BA^T\), and the minimum can be evaluated explicitly.  Although this aligned error is useful for diagnostic purposes, the intrinsic distance and the projection-matrix error are usually more natural for Grassmann-valued interpolation.  If the entries of \(Y\) themselves carry physical meaning, rather than only the subspace \(\spanop(Y)\), then the problem should be formulated on a different matrix manifold, such as the Stiefel manifold or a fixed-rank matrix manifold.

\section{Comparison of Several Approaches}
\label{sec:comparison_three_methods}

We compare three approaches for interpolating Grassmann-valued data: the recursive geodesic construction \(\GIDER_n\), tangent-space interpolation \(\mathrm{TSI}_n\), and projection-matrix interpolation followed by rank-\(r\) projection, denoted by \(\mathrm{PROJ}_n\).  All three methods are invariant under changes of orthonormal basis and are therefore appropriate for data on \(\Gr(r,m)\).  They differ, however, in the representation in which interpolation is performed.  The method \(\GIDER_n\) is recursive and geodesic, \(\mathrm{TSI}_n\) is a single-chart tangent-space method, and \(\mathrm{PROJ}_n\) is an extrinsic method based on the projector embedding of the Grassmann manifold.  The distinction is important because these methods may have the same local order of accuracy for smooth and densely sampled data while having different numerical costs, auxiliary choices, and robustness properties.

Let \(\mathcal{Y}_0,\ldots,\mathcal{Y}_n\in\Gr(r,m)\) be sampled at \(t_0,\ldots,t_n\).  The proposed \(\GIDER_n\) method is defined by a recursive Neville-type construction in which each affine interpolation step is replaced by \(\GLERP\).  One sets
\[
    \mathcal{I}_j^0(t)=\mathcal{Y}_j
\]
and recursively defines
\[
    \mathcal{I}_j^k(t)
    =
    \GLERP\left(
        \mathcal{I}_j^{k-1}(t),
        \mathcal{I}_{j+1}^{k-1}(t);
        \frac{t-t_j}{t_{j+k}-t_j}
    \right),
    \qquad k=1,\ldots,n.
\]
The final interpolant is \(\mathcal{I}_0^n(t)\).  Thus \(\GIDER_n\) constructs the interpolant through successive geodesic interpolations between intermediate Grassmann points.  No fixed reference tangent space is chosen; the construction moves recursively on the manifold.  This makes \(\GIDER_n\) a direct Grassmann analogue of Neville interpolation.

Tangent-space interpolation proceeds differently.  It first selects a reference subspace \(\mathcal{Y}_{\rm ref}\), for example one endpoint, the middle point of the stencil, or the data point closest to the target parameter.  The stencil data are then mapped to the tangent space \(T_{\mathcal{Y}_{\rm ref}}\Gr(r,m)\) by
\[
    \Delta_j=\Log_{\mathcal{Y}_{\rm ref}}(\mathcal{Y}_j),
    \qquad j=0,\ldots,n.
\]
Ordinary Lagrange interpolation is performed in this vector space:
\[
    \Delta(t)=\sum_{j=0}^n L_j(t)\Delta_j,
    \qquad
    L_j(t)=\prod_{\ell\ne j}\frac{t-t_\ell}{t_j-t_\ell},
\]
and the result is mapped back to the Grassmann manifold by
\[
    \widehat{\mathcal{Y}}_{\rm TSI}(t)
    =
    \Exp_{\mathcal{Y}_{\rm ref}}\bigl(\Delta(t)\bigr).
\]
The main advantage of \(\mathrm{TSI}_n\) is its simplicity.  After the reference tangent space has been chosen, the problem reduces to ordinary vector-valued polynomial interpolation.  The method is also efficient when many target points are evaluated using the same stencil, since the logarithms of the data can be computed once and reused.

The principal limitation of \(\mathrm{TSI}_n\) is its dependence on the reference subspace.  Choosing \(\mathcal{Y}_0\), choosing a midpoint of the stencil, or choosing a target-dependent reference generally gives different interpolants.  These differences are typically higher order when all data are sufficiently close, but they can become visible for wider stencils, strongly curved curves, or data near the boundary of a normal neighborhood.  The method also requires all stencil points to lie inside the normal neighborhood of \(\mathcal{Y}_{\rm ref}\), and the tangent polynomial \(\Delta(t)\) may leave the region where the chosen normal coordinates accurately represent the manifold geometry.

Projection-matrix interpolation uses the embedding of \(\Gr(r,m)\) into the space of symmetric matrices.  If \(\mathcal{Y}_j=\spanop(Y_j)\), then the orthogonal projector
\[
    P_j=Y_jY_j^T
\]
is independent of the choice of orthonormal basis for the subspace.  One interpolates the projectors in the ambient vector space:
\[
    \widetilde P(t)=\sum_{j=0}^n L_j(t)P_j.
\]
The matrix \(\widetilde P(t)\) is symmetric but is generally not an orthogonal projector and need not have rank \(r\).  The method therefore returns to \(\Gr(r,m)\) by taking the leading \(r\) eigenvectors of \(\widetilde P(t)\).  If \(U_r(t)\in\bR^{m\times r}\) contains these eigenvectors, then
\[
    \widehat{\mathcal{Y}}_{\rm PROJ}(t)=\spanop(U_r(t)).
\]
This method is easy to implement and avoids Grassmann logarithm maps.  It is, however, extrinsic: interpolation is performed in the ambient linear space of symmetric matrices rather than by following geodesics on \(\Gr(r,m)\).  The final rank-\(r\) projection is stable when there is a clear spectral gap between the \(r\)-th and \((r+1)\)-st eigenvalues of \(\widetilde P(t)\), but it can be sensitive when this gap is small.

The three methods are locally consistent in a smooth small-data regime.  Suppose that all data lie in a normal neighborhood of a point \(q\), and write
\[
    \mathcal{Y}_j=\Exp_q(u_j),
    \qquad
    u_j\in T_q\Gr(r,m),
    \qquad
    \norm{u_j}=\mathcal{O}(h).
\]
In these local coordinates, geodesic interpolation has the expansion
\[
    \Log_q\left(
        \GLERP(\Exp_q u,\Exp_q v;\lambda)
    \right)
    =
    (1-\lambda)u+\lambda v+\mathcal{R}(u,v;\lambda),
\]
where \(\mathcal{R}\) is a higher-order curvature-dependent remainder.  Hence \(\GIDER_n\) agrees, to leading order, with ordinary Neville interpolation of the coordinate data.  Tangent-space interpolation also performs polynomial interpolation in a tangent space.  Projection-matrix interpolation is locally consistent because the map \(\mathcal{Y}\mapsto YY^T\) is a smooth embedding of \(\Gr(r,m)\) into the space of symmetric matrices.  Under appropriate local regularity assumptions, all three methods therefore have the same expected order, namely order \(n+1\) for a sufficiently smooth Grassmann-valued curve.  The numerical experiments in Section~\ref{sec:numerical_example} illustrate this behavior.

The close agreement between \(\GIDER_n\) and \(\mathrm{TSI}_n\) in smooth examples is therefore not surprising.  Both methods reduce locally to polynomial interpolation in a vector-space representation.  The difference is how this vector-space behavior is realized.  The method \(\mathrm{TSI}_n\) realizes it explicitly by choosing one tangent space and interpolating there.  The method \(\GIDER_n\) realizes it implicitly through repeated geodesic interpolation between intermediate subspaces.  Thus the two methods have the same leading-order behavior in a local chart, but not the same construction.

This distinction becomes important outside the ideal local regime.  Tangent-space interpolation may be preferable when the data remain close to a well-chosen reference subspace and computational efficiency is important.  It is particularly convenient for repeated evaluations from the same stencil, because the tangent-space representation of the data can be reused.  Its disadvantages are the dependence on the reference point and the reliance on one normal coordinate chart.  If the stencil is wide or the curve is strongly curved, the reference point may be far from part of the data, and the tangent polynomial may no longer describe the manifold geometry accurately.

The strengths and weaknesses of \(\GIDER_n\) are complementary.  Its main advantage is that it does not commit to one tangent space.  Every step is a geodesic interpolation between two Grassmann points, so the method is geometrically recursive and closely mirrors the replacement of Euclidean line segments by geodesic segments.  This structure is especially suitable for an ENO-type extension: different candidate stencils can be used to construct complete Grassmann-valued curves, and their oscillation can be compared using an intrinsic variation measure.  The cost is that \(\GIDER_n\) generally requires more exponential and logarithm evaluations than \(\mathrm{TSI}_n\).  Moreover, because Neville-type interpolation may use parameters outside \([0,1]\), the recursive construction can involve geodesic extrapolation.  The original data points and the intermediate subspaces generated by the recursion must therefore remain in a region where the required logarithms are well-defined and numerically stable.

The method \(\mathrm{PROJ}_n\) provides a useful extrinsic baseline.  It avoids choosing a tangent reference point and avoids repeated geodesic computations.  It is also naturally basis-invariant because projectors represent subspaces rather than bases.  These features make it attractive as a simple implementation.  Its main drawback is that the interpolation is not performed on the manifold itself.  The intermediate matrix \(\widetilde P(t)\) may have eigenvalues outside \([0,1]\), and the final step of taking leading eigenvectors introduces a nonlinear spectral projection.  Consequently, \(\mathrm{PROJ}_n\) can have the correct local order while having a different error constant and a sensitivity determined by the eigenvalue gap of \(\widetilde P(t)\).

The main distinctions are summarized in Table~\ref{tab:method_comparison}.

\begin{table}[t]
\centering
\caption{Comparison of three Grassmann interpolation strategies.}
\label{tab:method_comparison}
\small
\resizebox{\textwidth}{!}{%
\begin{tabular}{p{0.20\textwidth}|p{0.24\textwidth}|p{0.24\textwidth}|p{0.24\textwidth}}
\toprule
 & \(\GIDER_n\) & \(\mathrm{TSI}_n\) & \(\mathrm{PROJ}_n\) \\
\midrule
Basic idea
&
Recursive Neville interpolation with each affine step replaced by \(\GLERP\).
&
Map all stencil data to one tangent space, interpolate there, and map back by \(\Exp\).
&
Interpolate orthogonal projectors and project the result back to rank \(r\).
\\[2mm]
Geometry
&
Intrinsic and geodesic at each recursive step.
&
Intrinsic after a reference tangent space has been chosen.
&
Basis-invariant but extrinsic.
\\[2mm]
Reference choice
&
No single reference tangent space is required.
&
Requires a reference subspace \(\mathcal{Y}_{\rm ref}\).
&
No tangent-space reference is required.
\\[2mm]
Return to \(\Gr(r,m)\)
&
Automatic through recursive \(\GLERP\) evaluations.
&
Through \(\Exp_{\mathcal{Y}_{\rm ref}}\).
&
Through leading-eigenvector rank-\(r\) projection.
\\[2mm]
Computational cost
&
Requires multiple geodesic interpolation steps at each target point.
&
Efficient once the stencil has been mapped to the reference tangent space.
&
Requires ambient interpolation and an eigenvalue decomposition.
\\[2mm]
Expected local order
&
Order \(n+1\), under a normal-neighborhood assumption.
&
Order \(n+1\), under a normal-neighborhood assumption.
&
Order \(n+1\) in a regular local regime with a stable spectral projection.
\\[2mm]
Main numerical risk
&
Recursive geodesic extrapolations may leave the normal neighborhood.
&
The tangent polynomial may leave the valid region of the chosen chart.
&
The rank-\(r\) projection may be sensitive when the spectral gap is small.
\\[2mm]
Main advantage
&
Recursive geodesic structure and natural ENO extension.
&
Simplicity and efficiency for repeated evaluations.
&
Simple basis-invariant extrinsic implementation.
\\[2mm]
ENO extension
&
Natural, since each candidate is an intrinsic recursive curve.
&
Possible, but stencil comparisons may depend on reference choices.
&
Possible, but the variation criterion is extrinsic unless reformulated intrinsically.
\\
\bottomrule
\end{tabular}%
}
\end{table}

Consequently, \(\mathrm{TSI}_n\) and \(\mathrm{PROJ}_n\) are useful comparison methods.  Tangent-space interpolation is the closest intrinsic baseline and explains why its numerical results can be nearly identical to those of \(\GIDER_n\) for dense smooth data.  Projection-matrix interpolation gives a simple extrinsic benchmark that is often robust and easy to implement.  The advantage of \(\GIDER_n\) is not that it necessarily produces a different asymptotic order in smooth local tests, but that it provides a recursive geodesic structure.  This structure mirrors Neville interpolation without committing to a single tangent space and gives a natural foundation for Grassmann ENO interpolation.

\section{Numerical Example}\label{sec:numerical_example}

We present a numerical experiment to verify the convergence behavior of the proposed \(\GIDER_n\) interpolants and to compare them with tangent-space interpolation, denoted by \(\mathrm{TSI}_n\), and projection-matrix interpolation followed by rank-\(r\) projection, denoted by \(\mathrm{PROJ}_n\).  The test curve is known at arbitrary parameter values, so the interpolation error can be evaluated directly.

We consider the Grassmann manifold \(\Gr(2,5)\).  Let \(e_1,\ldots,e_5\) be the standard basis of \(\bR^5\), and set \(Y_\ast=[e_1,e_2]\in\bR^{5\times 2}\).  We generate a smooth curve of two-dimensional subspaces by applying a time-dependent orthogonal transformation to this reference subspace.  Let \(E_{ij}\) denote the \(5\times 5\) matrix with entry \(1\) in position \((i,j)\) and zero elsewhere.  Define the skew-symmetric matrices
\[
    \Omega_1=(E_{31}-E_{13})+\frac{1}{2}(E_{42}-E_{24}),\qquad
    \Omega_2=(E_{41}-E_{14})+(E_{52}-E_{25}),
\]
and
\[
    \Omega_3=(E_{51}-E_{15})+\frac{7}{10}(E_{32}-E_{23}).
\]
For \(t\in[0,1]\), define
\[
    \Omega(t)
    =
    0.25\sin(2\pi t)\Omega_1
    +
    0.20\cos(3\pi t)\Omega_2
    +
    0.15t\,\Omega_3.
\]
Since \(\Omega(t)\) is skew-symmetric, \(Q(t)=\exp(\Omega(t))\) is orthogonal.  The exact Grassmann-valued curve is then
\[
    \gamma(t)=\spanop(Y(t)),\qquad
    Y(t)=Q(t)Y_\ast.
\]
The columns of \(Y(t)\) are orthonormal for all \(t\).  The curve is smooth and non-geodesic, and is therefore suitable for observing the interpolation orders.

For a positive integer \(N\), let \(h=1/N\) and sample the exact curve at the uniform grid \(t_i=ih\), \(i=0,\ldots,N\).  The data are \(\mathcal{Y}_i=\gamma(t_i)\).  For each order \(n=1,2,3,4\), we compare three methods.  The first method is the recursive geodesic interpolant \(\GIDER_n\).  The second method, \(\mathrm{TSI}_n\), maps the stencil data to one reference tangent space, applies ordinary Lagrange interpolation there, and maps the result back by the exponential map.  The third method, \(\mathrm{PROJ}_n\), interpolates the projection matrices \(P_i=Y_iY_i^T\) entrywise and then returns to \(\Gr(2,5)\) by taking the leading two eigenvectors of the interpolated symmetric matrix.

On each interval \([t_i,t_{i+1}]\), the reconstruction uses the forward stencil \(\{\mathcal{Y}_i,\mathcal{Y}_{i+1},\ldots,\mathcal{Y}_{i+n}\}\) whenever \(i+n\le N\).  Near the right boundary, we use the last available stencil \(\{\mathcal{Y}_{N-n},\ldots,\mathcal{Y}_N\}\).  The error is evaluated on a refined set of test points.  On each coarse interval, we take \(K=10\) equally spaced subintervals and evaluate both the exact and interpolated subspaces at \(t_{i,\ell}=t_i+\ell h/K\), \(\ell=0,\ldots,K\).

The intrinsic pointwise error is measured by the Grassmann distance
\[
    e_{\Gr}(t_{i,\ell})
    =
    d_{\Gr}\bigl(Y(t_{i,\ell}),\widehat Y_h(t_{i,\ell})\bigr),
\]
where the principal angles are obtained from the singular values of \(Y(t_{i,\ell})^T\widehat Y_h(t_{i,\ell})\).  The maximum intrinsic error is
\[
    E_{\Gr}(h)
    =
    \max_{i,\ell} e_{\Gr}(t_{i,\ell}).
\]
For comparison, we also report the projection-matrix error.  Given the exact and interpolated orthonormal representatives \(Y(t)\) and \(\widehat Y_h(t)\), respectively, define
\[
    e_P(t)
    =
    \left\|
        Y(t)Y(t)^T
        -
        \widehat Y_h(t)\widehat Y_h(t)^T
    \right\|_F.
\]
This quantity is invariant under changes of basis \(Y\mapsto YR\), \(R\in O(r)\), and is therefore a valid Grassmann error measure.  If \(\theta_1,\ldots,\theta_r\) are the principal angles, then
\[
    \frac{1}{\sqrt{2}}e_P(t)
    =
    \left(\sum_{j=1}^r\sin^2\theta_j(t)\right)^{1/2}.
\]
Thus \(e_P(t)/\sqrt{2}\) is asymptotically equivalent to the intrinsic Grassmann distance for small errors.  The maximum projection error is defined by
\[
    E_P(h)=\max_{i,\ell} e_P(t_{i,\ell}).
\]
The observed convergence rate is computed by
\[
    \rho_h
    =
    \log_2\left(\frac{E(h)}{E(h/2)}\right).
\]

Table \ref{tab:comparison_intrinsic} reports the errors measured by the intrinsic Grassmann distance.  The results show the expected convergence rates: order \(n+1\) for each of the three methods in the pre-asymptotic and asymptotic regimes.  For the highest-order methods on the finest grids, the intrinsic-distance error becomes sensitive to roundoff because the computation of the principal angles uses \(\theta_j=\arccos(\sigma_j)\), where \(\sigma_j\) is extremely close to one.  This explains the degraded rates in a few entries of the intrinsic-distance table.

\begin{table}[t]
\centering
\caption{Maximum error measured by the intrinsic Grassmann distance.}
\label{tab:comparison_intrinsic}
\small
\resizebox{\textwidth}{!}{
\begin{tabular}{cc|cc|cc|cc}
\toprule
Order & \(N\)
&
\(E_{\GIDER}\) & rate
&
\(E_{\mathrm{TSI}}\) & rate
&
\(E_{\mathrm{PROJ}}\) & rate
\\
\midrule
1 & 16  & \(1.2843{\rm e}{-02}\) & --     & \(1.2843{\rm e}{-02}\) & --     & \(1.2843{\rm e}{-02}\) & --     \\
1 & 32  & \(3.2596{\rm e}{-03}\) & 1.9782 & \(3.2596{\rm e}{-03}\) & 1.9782 & \(3.2596{\rm e}{-03}\) & 1.9782 \\
1 & 64  & \(8.1393{\rm e}{-04}\) & 2.0017 & \(8.1393{\rm e}{-04}\) & 2.0017 & \(8.1393{\rm e}{-04}\) & 2.0017 \\
1 & 128 & \(2.0392{\rm e}{-04}\) & 1.9969 & \(2.0392{\rm e}{-04}\) & 1.9969 & \(2.0392{\rm e}{-04}\) & 1.9969 \\
1 & 256 & \(5.1006{\rm e}{-05}\) & 1.9992 & \(5.1006{\rm e}{-05}\) & 1.9992 & \(5.1006{\rm e}{-05}\) & 1.9992 \\
\midrule
2 & 16  & \(3.7682{\rm e}{-03}\) & --     & \(3.7710{\rm e}{-03}\) & --     & \(4.6817{\rm e}{-03}\) & --     \\
2 & 32  & \(4.7903{\rm e}{-04}\) & 2.9757 & \(4.7921{\rm e}{-04}\) & 2.9762 & \(6.0778{\rm e}{-04}\) & 2.9454 \\
2 & 64  & \(6.0127{\rm e}{-05}\) & 2.9940 & \(6.0138{\rm e}{-05}\) & 2.9943 & \(7.6714{\rm e}{-05}\) & 2.9860 \\
2 & 128 & \(7.5233{\rm e}{-06}\) & 2.9986 & \(7.5240{\rm e}{-06}\) & 2.9987 & \(9.6125{\rm e}{-06}\) & 2.9965 \\
2 & 256 & \(9.4125{\rm e}{-07}\) & 2.9987 & \(9.4102{\rm e}{-07}\) & 2.9992 & \(1.2023{\rm e}{-06}\) & 2.9991 \\
\midrule
3 & 16  & \(1.3831{\rm e}{-03}\) & --     & \(1.3658{\rm e}{-03}\) & --     & \(1.6919{\rm e}{-03}\) & --     \\
3 & 32  & \(8.8460{\rm e}{-05}\) & 3.9668 & \(8.7738{\rm e}{-05}\) & 3.9604 & \(1.1211{\rm e}{-04}\) & 3.9156 \\
3 & 64  & \(5.5740{\rm e}{-06}\) & 3.9882 & \(5.5159{\rm e}{-06}\) & 3.9915 & \(7.2194{\rm e}{-06}\) & 3.9569 \\
3 & 128 & \(3.4915{\rm e}{-07}\) & 3.9968 & \(3.4563{\rm e}{-07}\) & 3.9963 & \(4.5442{\rm e}{-07}\) & 3.9898 \\
3 & 256 & \(4.2147{\rm e}{-08}\) & 3.0503 & \(3.9425{\rm e}{-08}\) & 3.1321 & \(4.2147{\rm e}{-08}\) & 3.4305 \\
\midrule
4 & 16  & \(5.8983{\rm e}{-04}\) & --     & \(5.7253{\rm e}{-04}\) & --     & \(1.5899{\rm e}{-03}\) & --     \\
4 & 32  & \(1.9561{\rm e}{-05}\) & 4.9142 & \(1.8811{\rm e}{-05}\) & 4.9277 & \(6.1610{\rm e}{-05}\) & 4.6897 \\
4 & 64  & \(6.2033{\rm e}{-07}\) & 4.9788 & \(5.9455{\rm e}{-07}\) & 4.9836 & \(2.0406{\rm e}{-06}\) & 4.9161 \\
4 & 128 & \(3.9425{\rm e}{-08}\) & 3.9759 & \(4.2147{\rm e}{-08}\) & 3.8183 & \(6.4953{\rm e}{-08}\) & 4.9735 \\
4 & 256 & \(4.2147{\rm e}{-08}\) & -0.0963 & \(3.9425{\rm e}{-08}\) & 0.0963 & \(3.9425{\rm e}{-08}\) & 0.7203 \\
\bottomrule
\end{tabular}
}
\end{table}

Table \ref{tab:comparison_projection} reports the same experiment using the projection-matrix error.  This error gives a cleaner view of the asymptotic behavior, especially for the higher-order methods.  The results confirm second-order convergence for \(n=1\), third-order convergence for \(n=2\), fourth-order convergence for \(n=3\), and fifth-order convergence for \(n=4\).  The projection-matrix results also show that \(\GIDER_n\) and \(\mathrm{TSI}_n\) have nearly identical errors in this smooth local test.  The \(\mathrm{PROJ}_n\) method achieves the same formal order but has a somewhat larger error constant, especially for \(n=4\).

\begin{table}[t]
\centering
\caption{Maximum error measured by the projection-matrix distance \(\|YY^T-\widehat Y\widehat Y^T\|_F\).}
\label{tab:comparison_projection}
\small
\resizebox{\textwidth}{!}{
\begin{tabular}{cc|cc|cc|cc}
\toprule
Order & \(N\)
&
\(E_{\GIDER}\) & rate
&
\(E_{\mathrm{TSI}}\) & rate
&
\(E_{\mathrm{PROJ}}\) & rate
\\
\midrule
1 & 16  & \(1.8163{\rm e}{-02}\) & --     & \(1.8163{\rm e}{-02}\) & --     & \(1.8163{\rm e}{-02}\) & --     \\
1 & 32  & \(4.6098{\rm e}{-03}\) & 1.9782 & \(4.6098{\rm e}{-03}\) & 1.9782 & \(4.6098{\rm e}{-03}\) & 1.9782 \\
1 & 64  & \(1.1511{\rm e}{-03}\) & 2.0017 & \(1.1511{\rm e}{-03}\) & 2.0017 & \(1.1511{\rm e}{-03}\) & 2.0017 \\
1 & 128 & \(2.8838{\rm e}{-04}\) & 1.9969 & \(2.8838{\rm e}{-04}\) & 1.9969 & \(2.8838{\rm e}{-04}\) & 1.9969 \\
1 & 256 & \(7.2133{\rm e}{-05}\) & 1.9992 & \(7.2133{\rm e}{-05}\) & 1.9992 & \(7.2133{\rm e}{-05}\) & 1.9992 \\
\midrule
2 & 16  & \(5.3290{\rm e}{-03}\) & --     & \(5.3329{\rm e}{-03}\) & --     & \(6.6209{\rm e}{-03}\) & --     \\
2 & 32  & \(6.7746{\rm e}{-04}\) & 2.9757 & \(6.7771{\rm e}{-04}\) & 2.9762 & \(8.5953{\rm e}{-04}\) & 2.9454 \\
2 & 64  & \(8.5033{\rm e}{-05}\) & 2.9940 & \(8.5048{\rm e}{-05}\) & 2.9943 & \(1.0849{\rm e}{-04}\) & 2.9860 \\
2 & 128 & \(1.0640{\rm e}{-05}\) & 2.9986 & \(1.0641{\rm e}{-05}\) & 2.9987 & \(1.3594{\rm e}{-05}\) & 2.9965 \\
2 & 256 & \(1.3303{\rm e}{-06}\) & 2.9997 & \(1.3303{\rm e}{-06}\) & 2.9997 & \(1.7003{\rm e}{-06}\) & 2.9991 \\
\midrule
3 & 16  & \(1.9560{\rm e}{-03}\) & --     & \(1.9316{\rm e}{-03}\) & --     & \(2.3927{\rm e}{-03}\) & --     \\
3 & 32  & \(1.2510{\rm e}{-04}\) & 3.9668 & \(1.2408{\rm e}{-04}\) & 3.9604 & \(1.5855{\rm e}{-04}\) & 3.9156 \\
3 & 64  & \(7.8827{\rm e}{-06}\) & 3.9883 & \(7.8008{\rm e}{-06}\) & 3.9915 & \(1.0210{\rm e}{-05}\) & 3.9569 \\
3 & 128 & \(4.9316{\rm e}{-07}\) & 3.9986 & \(4.8857{\rm e}{-07}\) & 3.9970 & \(6.4265{\rm e}{-07}\) & 3.9898 \\
3 & 256 & \(3.0841{\rm e}{-08}\) & 3.9991 & \(3.0557{\rm e}{-08}\) & 3.9990 & \(4.0208{\rm e}{-08}\) & 3.9985 \\
\midrule
4 & 16  & \(8.3414{\rm e}{-04}\) & --     & \(8.0968{\rm e}{-04}\) & --     & \(2.2485{\rm e}{-03}\) & --     \\
4 & 32  & \(2.7664{\rm e}{-05}\) & 4.9142 & \(2.6603{\rm e}{-05}\) & 4.9277 & \(8.7129{\rm e}{-05}\) & 4.6897 \\
4 & 64  & \(8.7740{\rm e}{-07}\) & 4.9786 & \(8.4053{\rm e}{-07}\) & 4.9841 & \(2.8860{\rm e}{-06}\) & 4.9160 \\
4 & 128 & \(2.7538{\rm e}{-08}\) & 4.9937 & \(2.6322{\rm e}{-08}\) & 4.9969 & \(9.1537{\rm e}{-08}\) & 4.9786 \\
4 & 256 & \(8.6166{\rm e}{-10}\) & 4.9982 & \(8.2277{\rm e}{-10}\) & 4.9996 & \(2.8713{\rm e}{-09}\) & 4.9946 \\
\bottomrule
\end{tabular}
}
\end{table}

The results support the theoretical accuracy statement for the proposed \(\GIDER_n\) construction.  In the projection-matrix error, the observed orders are very close to \(n+1\) for all tested values of \(n\).  The comparison also illustrates the close relationship between \(\GIDER_n\) and tangent-space interpolation in a smooth local regime.  Their errors are nearly indistinguishable, which is consistent with the fact that both methods reduce locally to polynomial interpolation in tangent coordinates.  The projection-matrix method also achieves the expected order, but with larger constants in this example.  This is consistent with its extrinsic construction: it interpolates in the ambient space of symmetric matrices and then applies a rank-\(r\) spectral projection, rather than building the curve directly from Grassmann geodesics.

The intrinsic Grassmann-distance table should be interpreted with some care for the smallest errors.  When the exact and interpolated subspaces are extremely close, the singular values of \(Y^T\widehat Y\) are nearly one.  Computing the principal angles by \(\theta_j=\arccos(\sigma_j)\) then becomes sensitive to floating-point roundoff.  This sensitivity explains the loss of clean fifth-order behavior in the final rows of Table \ref{tab:comparison_intrinsic}.  The projection-matrix error avoids this inverse-cosine sensitivity and therefore gives a more reliable diagnostic of the high-order convergence in the finest-grid regime.

\section{Toward Grassmann ENO Interpolation}\label{sec:geno}

The recursive \(\GIDER_n\) construction provides a family of high-order interpolants on the Grassmann manifold.  As in Euclidean and spherical interpolation, however, a high-order reconstruction may become oscillatory when the underlying curve has a kink, a sharp turn, or a loss of smoothness.  It is therefore natural to combine the recursive Grassmann interpolation with the essentially non-oscillatory philosophy.  We refer to the resulting idea as Grassmann ENO interpolation, or GENO.

Suppose that one wants to reconstruct the curve between two consecutive data points \(\mathcal{Y}_i\) and \(\mathcal{Y}_{i+1}\).  From neighboring data, one constructs several candidate \(\GIDER_n\) curves on different stencils, each of which interpolates \(\mathcal{Y}_i\) and \(\mathcal{Y}_{i+1}\) but uses different additional points to the left or right.  The GENO selection principle is to choose the candidate whose intrinsic variation on the target interval is smallest.  This follows the same rationale as the spherical SENO construction in \cite{FongLeung2023}: a candidate that crosses a nonsmooth region tends to have a larger geometric variation and is therefore less desirable.

Let \(\widehat\gamma(t)\) be one candidate reconstruction on the interval \([t_i,t_{i+1}]\), and let \(s_0,\ldots,s_K\) be a fine partition of this interval.  Its discrete Grassmann variation is defined by
\[
    \mathcal{V}(\widehat\gamma)
    =
    \sum_{\ell=0}^{K-1}
    d_{\Gr}\bigl(
        \widehat\gamma(s_\ell),
        \widehat\gamma(s_{\ell+1})
    \bigr).
\]
Equivalently, if \(\widehat P(t)\) denotes the orthogonal projector associated with \(\widehat\gamma(t)\), one may use the projection-based variation
\[
    \mathcal{V}_P(\widehat\gamma)
    =
    \frac{1}{\sqrt{2}}
    \sum_{\ell=0}^{K-1}
    \norm{
        \widehat P(s_\ell)
        -
        \widehat P(s_{\ell+1})
    }_F.
\]
For small principal angles, \(\mathcal{V}\) and \(\mathcal{V}_P\) are locally equivalent.  The projection-based form is particularly simple to implement, while the geodesic-distance form is more directly intrinsic.

More explicitly, let \(\mathcal{S}_q\) denote the admissible stencils that contain both endpoints \(\mathcal{Y}_i\) and \(\mathcal{Y}_{i+1}\), and let \(\widehat\gamma_q\) be the \(\GIDER_n\) candidate generated from stencil \(\mathcal{S}_q\).  The GENO interpolant is defined by
\[
    \widehat\gamma_{\mathrm{GENO}}
    =
    \widehat\gamma_{q^\ast},
    \qquad
    q^\ast
    =
    \argmin_q \mathcal{V}(\widehat\gamma_q),
\]
or, equivalently in implementation, by minimizing \(\mathcal{V}_P(\widehat\gamma_q)\).  Since every candidate is a \(\GIDER_n\) reconstruction on a valid stencil, the selected curve remains Grassmann-valued and interpolates the two endpoints.  When the underlying curve is smooth across all candidate stencils, the variations are comparable and the selected candidate retains the high-order accuracy of \(\GIDER_n\).  When a kink or sharp turn lies within some candidate stencils, the variation criterion is designed to avoid those candidates and to select a less oscillatory reconstruction.

As with \(\GIDER_n\), the GENO construction is local.  All candidate reconstructions and all intermediate geodesic extrapolations must remain in a common normal neighborhood.  In practice, this requires sufficiently dense sampling and relevant principal angles bounded away from \(\pi/2\).  Under this admissibility condition, the GENO procedure gives a direct Grassmann analogue of SENO: high-order interpolation is retained in smooth regions, while the stencil selection suppresses spurious oscillations near nonsmooth features.

\section{Conclusion}
\label{sec:conclusion}

We have developed a Grassmann-manifold analogue of the \(\SLERP\)--based interpolation philosophy.  The basic interpolant, \(\GLERP\), is defined by the Grassmann exponential and logarithm maps and follows the constant-speed geodesic between two nearby subspaces.  It is second-order accurate for smooth Grassmann-valued curves under a standard normal-neighborhood assumption.  Higher-order interpolants are obtained through the recursive \(\GIDER_n\) construction, in which each affine step of Neville interpolation is replaced by \(\GLERP\).  The resulting method interpolates all \(n+1\) data subspaces exactly and is locally of order \(n+1\) for sufficiently smooth curves.

A central point is that matrix data representing subspaces should be interpolated and compared on the Grassmann manifold, not by entrywise operations on arbitrary matrix representatives.  Intrinsic Grassmann distance, projection-matrix error, and Procrustes-aligned basis error provide basis-invariant diagnostics.  The comparison with tangent-space interpolation and projection-matrix interpolation shows that all three approaches have the same expected local order in a regular smooth regime, while differing in geometry, computational cost, and robustness.  Tangent-space interpolation is simple and efficient once a reference tangent space is chosen, projection-matrix interpolation is an easily implemented extrinsic baseline, and \(\GIDER_n\) provides a recursive geodesic construction that avoids committing to a single tangent space.

The numerical example on \(\Gr(2,5)\) confirms the predicted convergence rates.  In smooth local tests, \(\GIDER_n\) and tangent-space interpolation are nearly indistinguishable in accuracy, while projection-matrix interpolation achieves the same order with somewhat larger constants in the reported example.  The recursive structure of \(\GIDER_n\) also suggests a natural route to Grassmann ENO interpolation, where candidate geodesic reconstructions are selected by minimizing intrinsic or projection-based variation.  Developing such GENO schemes for nonsmooth Grassmann-valued data and applying them to subspace-valued numerical reconstruction problems are natural directions for future work.

\section*{Computational Methodology and Disclosure}

This research utilized GPT-5.5 (Thinking) to perform primary computational modeling and mathematical derivation. The author acted as the principal investigator, defining the research parameters and theoretical scope. The author notes that while the computational workflow was executed via AI, the author maintains full responsibility for the study's conclusions. The author acknowledges that the mathematical depth of the generated results exceeds current manual verification capabilities and presents these findings as AI-assisted hypotheses subject to future formal peer verification. Consequently, this article is intended solely for dissemination as a preprint on arXiv and is not submitted to peer-reviewed journals in its current form.


\begin{thebibliography}{99}

\bibitem{FongLeung2023}
K. W. Fong and S. Leung,
\newblock Spherical essentially non-oscillatory (SENO) interpolation,
\newblock \emph{Journal of Scientific Computing}, 94:28, 2023.

\bibitem{EdelmanAriasSmith1998}
A. Edelman, T. A. Arias, and S. T. Smith,
\newblock The geometry of algorithms with orthogonality constraints,
\newblock \emph{SIAM Journal on Matrix Analysis and Applications}, 20(2):303--353, 1998.

\bibitem{AbsilMahonySepulchre2008}
P.-A. Absil, R. Mahony, and R. Sepulchre,
\newblock \emph{Optimization Algorithms on Matrix Manifolds},
\newblock Princeton University Press, 2008.


\bibitem{AmsallemFarhat2011}
D. Amsallem and C. Farhat,
\newblock An online method for interpolating linear parametric reduced-order models,
\newblock \emph{SIAM Journal on Scientific Computing}, 33(5):2169--2198, 2011.

\bibitem{MosqueraEtAl2019}
R. Mosquera, A. El Hamidi, A. Hamdouni, and A. Falaize,
\newblock Generalization of the Neville--Aitken interpolation algorithm on Grassmann manifolds: applications to reduced order model,
\newblock arXiv:1907.02831, 2019.

\bibitem{NoakesHeinzingerPaden1989}
L. Noakes, G. Heinzinger, and B. Paden,
\newblock Cubic splines on curved spaces,
\newblock \emph{IMA Journal of Mathematical Control and Information}, 6(4):465--473, 1989.

\end{thebibliography}
\end{document}